\documentclass[a4paper,10pt]{amsart}
\usepackage[colorlinks,linkcolor=blue,citecolor=blue]{hyperref}
\usepackage{latexsym, amssymb, amsmath, amscd, amsthm, mathrsfs, bbm}
\usepackage[all, knot]{xy}
\xyoption{arc}

\usepackage{anysize}\marginsize{22mm}{22mm}{25mm}{25mm}
\allowdisplaybreaks[4]

\def \k {\mathbbm{k}}
\def \End {\operatorname{End}}
\def \V {\operatorname{V}_{n,d}}

\def \dim {\operatorname{dim}}
\def \rk {\operatorname{Rank}}
\def \v {\operatorname{v}}
\def \GL {\operatorname{GL}}
\def \Im {\operatorname{Im}}
\def \Q {\mathbbm{Q}}

\numberwithin{equation}{section}
\numberwithin{table}{section}
\numberwithin{equation}{section}
\newtheorem{theorem}{Theorem}[section]
\newtheorem{lemma}[theorem]{Lemma}
\newtheorem{proposition}[theorem]{Proposition}
\newtheorem{corollary}[theorem]{Corollary}
\newtheorem{definition}[theorem]{Definition}
\newtheorem{example}[theorem]{Example}
\newtheorem{algorithm}[theorem]{Algorithm}
\newtheorem{remark}[theorem]{Remark}
\newtheorem{remarks}[theorem]{Remarks}

\title[Centers and Decompositions of Forms]{On Centers and Direct Sum Decompositions \\ of Higher Degree Forms}
\thanks{Supported by NSFC 11911530172, 11971181 and 11971449.}



\author{Hua-Lin Huang, Huajun Lu, Yu Ye and Chi Zhang}
\address{School of Mathematical Sciences, Huaqiao University, Quanzhou 362021, China}
\email{hualin.huang@hqu.edu.cn}

\address{School of Mathematical Sciences, Huaqiao University, Quanzhou 362021, China}
\email{huajun@hqu.edu.cn}

\address{School of Mathematical Sciences, University of Science and Technology of China, Hefei 230026, China}
\email{yeyu@ustc.edu.cn}

\address{Department of Mathematics, Northeastern University, Shenyang 110819, China}
\email{zhangchi@mail.neu.edu.cn}
\begin{document}

\begin{abstract}
Higher degree forms are homogeneous polynomials of degree $d > 2,$ or equivalently symmetric $d$-linear spaces. This paper is mainly concerned about the algebraic structure of the centers of higher degree forms with applications specifically to direct sum decompositions, namely expressing higher degree forms as sums of forms in disjoint sets of variables. We show that the center algebra of almost every form is the ground field, consequently almost all higher degree forms are absolutely indecomposable. If a higher degree form is decomposable, then we provide simple criteria and algorithms for direct sum decompositions by its center algebra. It is shown that the direct sum decomposition problem can be boiled down to some standard tasks of linear algebra, in particular the computations of eigenvalues and eigenvectors. We also apply the structure results of center algebras to provide a complete answer to the classical problem of whether a higher degree form can be reconstructed from its Jacobian ideal.
\end{abstract}

\maketitle

\section{Introduction}
One of the central problems of classical invariant theory is the equivalence of higher degree forms under linear changes of variables. The so-called direct sum decomposition of a higher degree form is expressing it as a sum of two or more forms in disjoint sets of variables, possibly after a reversible linear change of variables. This is a natural step of dimension reduction as it provides the separation of variables. Direct sum decomposition of higher degree forms also plays important roles in many problems of such subjects as commutative algebra, geometric invariant theory, multilinear algebra and computational complexity.

In his pioneering work \cite{h1}, Harrison initiated to generalize Witt's algebraic theory of quadratic forms to the higher degree situation. One of his important results is that any nondegenerate higher degree form admits a unique decomposition into a direct sum of indecomposable forms. In his research, the notion of centers of higher degree forms is crucial. Centers of higher degree forms, possibly in different terminologies, with connection to their direct sum decompositions were rediscovered independently several times by other authors, for example \cite{p, cg, r, k}.

Direct sum decompositions of homogeneous polynomials are also considered in the Sebastiani-Thom type problems of algebraic geometry, where a decomposable polynomial is called of Sebastiani-Thom type. In \cite{w}, some sufficient conditions are provided for the direct sum decomposability of homogeneous polynomials via their Jacobian ideals. Recently, direct sum decompositions of higher degree forms are approached by apolarity as well, see for example \cite{k, bbkt, sh}. In \cite{f1}, the direct sum decomposition of a smooth form is interpreted in terms of the product factorization of its associated form and an algorithm for computing direct sum decompositions is provided. In these aforementioned works, criteria and algorithms of direct sum decompositions of higher degree forms involve sophisticated tools and are computationally expensive.

The main aim of the present paper is to stress that the theory of Harrison's centers is highly effective for direct sum decompositions of higher degree forms both theoretically and computationally. There are quite a few research works (see e.g. \cite{h2,keet,os1,os2,pu,s}) along the line of \cite{h1}, however general algebraic structures of centers and algorithms for direct sum decompositions via centers were not systematically pursued before. We show that almost all forms have trivial center (i.e. isomorphic to the ground field), that means are \emph{a priori} absolutely indecomposable. We prove that the center algebra of a higher degree form is semisimple if and only if the form is not a limit of direct sums forms \cite{bbkt}. For forms with semisimple center, we give an elementary criterion for the direct sum decomposability which is equivalent to computing the rank of a finite set of vectors. It can be observed that smooth forms have semisimple centers, so our criterion includes and highly simplifies those cases treated in \cite{bbkt,f1,w}. Our structure results of center algebras can also be applied to provide a complete linear algebraic answer to the classical problem of whether higher degree forms can be reconstructed from their Jacobian ideal. The latter is very important in the Torelli problem, see e.g. \cite{cg, ron}. Moreover, we show through a simple algorithm that direct sum decompositions for arbitrary higher degree forms can be boiled down to some standard tasks of linear algebra, specifically the computations of eigenvalues and eigenvectors for which there are efficient algorithms and well-established softwares. Some complicated examples treated in \cite{f1, bbkt, sh}, partially involving computers, are easily rehandled by hand via this approach.

The rest of the paper is organized as follows. Section 2 is devoted to some preliminaries on centers and direct sum decompositions of higher degree forms. Main results are presented in Section 3. We conclude the paper with some examples in Section 4. Throughout, let $d \ge 3$ be an integer and let $\k$ be a field with $\operatorname{char} \k = 0,$ or $\operatorname{char} \k > d.$

\section{Preliminaries}
A form of degree $d$ in $n$ variables is an element of the polynomial ring $\k[x_1, \cdots, x_n]$ which is a sum of monomials of degree $d.$ A form $f$ is called a direct sum if, after an invertible linear change of variables, it can be written as a sum of $t \ge 2$ forms in disjoint sets of variables as follows
\begin{equation} \label{eds} f=f_1(x_1, \cdots, x_{a_1}) + \cdots + f_t(x_{a_{t-1}+1}, \cdots, x_n). \end{equation}
If this is not the case, then $f$ is said to be indecomposable. On the other extreme, if the $f_i$'s are forms in only one variable, then $f$ is said to be diagonalizable.

For convenience, a general form of degree $d$ in $n$ variables is written in the symmetric way:
\begin{equation} \label{et}
f(x_1, \cdots, x_n) = \sum_{1 \le i_1, \cdots, i_d \le n} a_{i_1 \cdots i_d} x_{i_1} \cdots x_{i_d}
\end{equation}
where the $a_{i_1 \cdots i_d}$'s are symmetric in the sense that they remain unchanged under any permutation of their subscripts. The resulting symmetric $d$-tensor $A=(a_{i_1 \cdots i_d})_{1 \le i_1, \cdots, i_d \le n}$ is called the symmetric tensor of $f.$ One can write the form $f(x_1, \cdots, x_n)=A x^d$ in terms of products of tensors, where $x=(x_1, \cdots, x_n)^T$ is the vector of variables. If $y=Px$ with $P=(p_{ij}) \in \GL(n,\k)$ is a change of variables, then the resulting form \[ g(y_1, \cdots, y_n)=\sum_{1 \le j_1, \cdots, j_d \le n} \ \sum_{1 \le i_1, \cdots, i_d \le n} a_{i_1 \cdots i_d}p_{i_1j_1} \cdots p_{i_dj_d}y_{j_1} \cdots y_{j_d} \] and the associated symmetric tensor becomes \[ AP^d:= \left( \sum_{1 \le i_1, \cdots, i_d \le n} a_{i_1 \cdots i_d}p_{i_1j_1} \cdots p_{i_dj_d} \right)_{1 \le j_1, \cdots, j_d \le n}.\] Call $AP^d$ the $d$-congruence of $A$ by $P.$ It is clear that $f$ is a direct sum if and only if its symmetric tensor $A$ is $d$-congruent to a block diagonal tensor with at least $2$ nonzero blocks, while $f$ is diagonalizable if and only if $A$ is $d$-congruent to a diagonal tensor. In this paper we mainly treat higher degree forms and leave the corresponding results of symmetric tensors to the interested reader.

Corresponding to forms of degree $d$ there are associated symmetric $d$-linear spaces. Let $V$ be a vector space over $\k$ of dimension $n$ with a basis $e_1, \cdots, e_n.$ Define $\Theta \colon V \times \cdots \times V \longrightarrow \k$ by $\Theta(e_{i_1}, \cdots, e_{i_d})=a_{i_1 \cdots i_d}.$ The pair $(V, \Theta)$ is called the associated symmetric $d$-linear space of $f$ under the basis $e_1, \cdots, e_n.$ One can recover the form $f$ from $(V, \Theta)$ as \[ f(x_1, \cdots, x_n) = \Theta\left(\sum_{1 \le i \le n}x_ie_i, \ \dots, \ \sum_{1 \le i \le n}x_i e_i\right).\] Nonzero subspaces $V_1, \cdots, V_t$ of $(V,\Theta)$ are said to be orthogonal, if $\Theta(v_1, \cdots, v_d)=0$ unless all the $v_i$'s are in the same $V_s$ for some $1 \le s \le t.$ If $V=V_1 \oplus \cdots \oplus V_t$ for $t \ge 2$ nonzero orthogonal subspaces, then call $(V,\Theta)$ decomposable. Otherwise, call $(V,\Theta)$ indecomposable. Clearly, the orthogonal decompositions of $(V,\Theta)$ are in bijection with the direct sum decompositions of its associated form $f.$

A form $f$ is said to be nondegenerate, if no variable can be removed by an invertible linear change of variables. This is equivalent to saying, in terms of symmetric $d$-linear spaces, that $\Theta(u, v_2, \cdots, v_d)=0$ for all $v_2, \cdots, v_d \in V$ implies $u=0.$ For the associated symmetric $d$-tensor $A,$ let $A_{i_1}$ denote the $(d-1)$-tensor $A=(a_{i_1 \cdots i_d})_{1 \le i_2, \cdots, i_d \le n}.$ Then $f$ is nondegenerate if and only if the $A_{i_1}$'s are linearly independent in the space of $(d-1)$-tensors. Moreover, the form $f$ involves essentially $\rk \{ A_1, \cdots, A_n \}$ variables. See \cite{hlyz} for more details.

According to \cite{h1, h2}, the center of a higher degree form can be defined in the following three equivalent ways. Suppose $f$ is a form of degree $d$ in $n$ variables. By $H$ we denote its Hessian matrix $(\frac{\partial^2 f}{\partial x_i \partial x_j})_{1 \le i,\ j \le n}.$ Then the center of $f$ is defined as
\begin{equation}\label{hessian}
 Z(f):=\{ X \in \k^{n \times n} \mid (HX)^T =  HX \}.
\end{equation}
Let $A=(a_{i_1 \cdots i_d})_{1 \le i_1, \cdots, i_d \le n}$ be the associated symmetric $d$-tensor of $f$ and $A^{(i_3 \cdots i_d)}$ the $n \times n$ matrix $(a_{i_1i_2i_3 \cdots i_d})_{1 \le i_1, i_2 \le n}.$ Then the center of $A$ is defined by
\begin{equation}\label{ec}
 Z(A):= \{ X \in \k^{n \times n} \mid X^TA^{(i_3 \cdots i_d)}=A^{(i_3 \cdots i_d)}X, \ \forall 1 \le i_3, \cdots, i_d \le n \}.
\end{equation}
In terms of the associated symmetric $d$-linear space, the center is defined as
\begin{equation}
Z(V,\Theta):=\{ \phi \in \operatorname{End}(V) \mid \Theta(\phi(v_1), v_2, \cdots, v_d)=\Theta(v_1, \phi(v_2), \cdots, v_d), \ \forall v_1, v_2, \cdots, v_d \in V \}.
\end{equation}

The following are some useful facts about centers and direct sum decompositions of higher degree forms obtained in \cite{h1}.

\begin{proposition} \label{pc}
Suppose $f$ is a nondegenerate form of degree $d$ in $n$ variables. Then
\begin{itemize}
\item[(1)] The center $Z(f)$ is a commutative subalgebra of the full matrix algebra $\k^{n \times n}.$
\item[(2)] There is a one-to-one correspondence between direct sum decompositions of $f$ and complete sets of orthogonal idempotents of $Z(f).$
\item[(3)] The decomposition of $f$ into a direct sum of indecomposable forms is unique up to equivalence and permutation of indecomposable summands.
\item[(4)] If $K/\k$ is a field extension, by $f_K$ it is meant treating $f \in K[x_1, \cdots, x_n],$ then $Z(f_K) \cong Z(f) \otimes_\k K.$
\end{itemize}
\end{proposition}

\begin{proof}
For later applications, we include a proof of item (2). Proofs of other items can be found in \cite{h1}. It is enough to prove the correspondence between direct sum decompositions into two terms and complete sets of pairwise orthogonal idempotents. This can be easily extended to the general situation.

Suppose $f(x_1, \cdots, x_n)=f_1(x_1, \cdots, x_a)+f_2(x_{a+1}, \cdots, x_n)$ is a direct sum decomposition. Let $(V, \Theta)$ be the associated symmetric $d$-linear space with basis $e_1, \cdots, e_n.$ Under these assumptions, we have
\begin{eqnarray*}
  f_1(x_1, \cdots, x_a) &=& \Theta\left(\sum_{i=1}^a x_ie_i, \cdots, \sum_{i=1}^a x_ie_i\right), \\
  f_2(x_{a+1}, \cdots, x_n) &=& \Theta\left(\sum_{i=a+1}^n x_ie_i, \cdots, \sum_{i=a+1}^n x_ie_i\right).
\end{eqnarray*}
Note in particular that
\begin{equation}\label{o}
\Theta(e_{i_1}, e_{i_2}, \cdots, e_{i_d})=0 \ \ \mathrm{unless} \ \ 1 \le i_1, i_2, \cdots, i_d \le a \ \ \mathrm{or} \ \ a+1 \le i_1, i_2, \cdots, i_d \le n.
\end{equation}
Let $V_1$ (resp. $V_2$) be the subspace of $V$ spanned by $e_1, \cdots, e_a$ (resp. $e_{a+1}, \cdots, e_n$) and let $\epsilon_i: V \twoheadrightarrow V_i$ be the natural projections. Then for each $v \in V,$ we have $v=\epsilon_1(v)+\epsilon_2(v).$ Clearly $1=\epsilon_1 + \epsilon_2.$ It remains to prove that $\epsilon_i \in Z(V,\Theta).$ Indeed, with \eqref{o} and the $d$-linearity of $\Theta$ we have
\begin{eqnarray*}
  \Theta(\epsilon_1(v_1), v_2, \cdots, v_d)
   &=& \Theta(\epsilon_1(v_1), \epsilon_1(v_2)+\epsilon_2(v_2), \cdots, \epsilon_1(v_d)+\epsilon_2(v_d)) \\
   &=& \Theta(\epsilon_1(v_1), \epsilon_1(v_2), \cdots, \epsilon_1(v_d)) \\
   &=& \Theta(\epsilon_1(v_1)+\epsilon_2(v_1), \epsilon_1(v_2), \cdots, \epsilon_1(v_d)+\epsilon_2(v_d)) \\
   &=& \Theta(v_1, \epsilon_1(v_2), \cdots, v_d).
\end{eqnarray*}
Similarly, we have $\epsilon_2 \in Z(V, \Theta).$

Conversely, suppose $\{\epsilon_1, \epsilon_2\}$ is a pair of orthogonal idempotents of the center $Z(V, \Theta)$ and $1=\epsilon_1 + \epsilon_2.$ Let $V_i=\epsilon_i(V).$ Then it is clear that $V = V_1 \oplus V_2.$ Assume that $e_1, \cdots, e_a$ (resp. $e_{a+1}, \cdots, e_n$) are a basis of $V_1$ (resp. $V_2$). As $\epsilon_1, \epsilon_2 \in Z(V, \Theta),$ it follows by the definition of centers that \eqref{o} holds. Now under the basis $e_1, \cdots, e_n$ we have
\begin{eqnarray*}
  f(x_1, \cdots, x_a, x_{a+1}, \cdots, x_n) &=& \Theta\left(\sum_{i=1}^n x_ie_i, \cdots, \sum_{i=1}^n x_ie_i\right) \\
   &=& \Theta\left(\sum_{i=1}^a x_ie_i, \cdots, \sum_{i=1}^a x_ie_i\right)+ \Theta\left(\sum_{i=a+1}^n x_ie_i, \cdots, \sum_{i=a+1}^n x_ie_i\right).
\end{eqnarray*}
This gives rise to a direct sum decomposition of $f.$
\end{proof}

\section{Main Results}
We are mainly concerned about the algebraic structure of the center algebra of an arbitrary higher degree form with applications to the problems of direct sum decomposition and reconstruction from Jacobian ideal. First of all, there is no loss of generality in assuming that the forms are nondegenerate. In order to take advantage of tools of algebraic geometry and for simplicity, we assume further that the ground field $\k$ is algebraically closed. For forms over an arbitrary ground field $\k,$ one may study their direct sum decompositions on the algebraic closure $\overline{\k}$ first. Then by Harrison's uniqueness result of decompositions, detect the exact situation on the original ground field directly. From now on, let $\V \subset \k[x_1, \cdots, x_n]$ denote the affine space of forms of degree $d$ in $n$ variables. The synonymous notions of higher degree forms and symmetric multilinear spaces are used interchangeably.

First we consider forms with trivial center. Recall that a higher degree form $f$ is called central if $Z(f) \cong \k.$ A form is called absolutely indecomposable if it remains indecomposable under every field extension of the ground field $\k.$ Thanks to items (2) and (4) of Proposition \ref{pc}, a central form is absolutely indecomposable. The Cayley canonical form $x_1^4+x_2^4+tx_1^2x_2^2$ implies that a general binary quartic is central, while the Hesse canonical form $x_1^3+x_2^3+x_3^3+6\lambda x_1x_2x_3$ implies that a general ternary cubic is central, see \cite[Examples 5.1 and 5.2]{hlyz} for explanations via their centers. It is reasonable to predict that a general higher degree form is central, except the simplest case of binary cubics. In fact, the center elements $(x_{ij}) \in \k^{2 \times 2}$ of a general binary cubic $ax_1^3+bx_1^2x_2+cx_1x_2^2+dx_2^3$ satisfy the following linear equations
\[
\left\{
  \begin{array}{ll}
    bx_{11}-3ax_{12}+cx_{21}-bx_{22}=0\\
    cx_{11}-bx_{12}+3dx_{21}-cx_{22}=0
  \end{array}
\right.
\]
so it follows that the center of a nondegenerate binary cubic is a $2$-dimensional commutative subalgebra of $\k^{2 \times 2},$ hence isomorphic to either $\k \times \k$ or $\k[x]/(x^2).$ Thus, there are no central binary cubics.

In the literature, there are few explicit examples of indecomposable higher degree forms for arbitrary $n$ and $d,$ much less central ones. It is generally believed to be very difficult to construct such examples, see e.g. \cite{pu, w}. In fact, with the help of centers we can easily provide explicit examples of central forms in $\V$ for arbitrary $(n, d) \neq (2,3).$

\begin{example}[A central form] \label{ecf}
As the Cayley and Hesse canonical forms already give central binary quartics and central ternary cubics, here we assume $d>4$ if $n=2$ and $n>3$ if $d=3.$ Let $f=x_1x_2^{d-1}+\cdots+ x_{n-1}x_n^{d-1}+x_nx_1^{d-1}$. Then $\frac{\partial^2 f}{\partial x_i \partial x_j}=0$ for all $1\leq i,j \leq n$ unless $$\frac{\partial^2 f}{\partial x_{j-1} \partial x_j}=(d-1)x_j^{d-2},\quad  \frac{\partial^2 f}{ \partial x_j^2}=(d-1)(d-2)x_{j-1}x_{j}^{d-3},\quad  \frac{\partial^2 f}{\partial x_{j+1} \partial x_j}=(d-1)x_{j+1}^{d-2}.$$ Here by abuse of notation we consider $n+1=1$ for subscripts. Let $P=(p_{ij}) \in Z(f)$. Then $HP$ is symmetric, and consequently we have
\begin{equation}\label{hp}
	p_{i-1,j}x_i^{d-2}+(d-2)p_{ij}x_{i-1}x_i^{d-3}+p_{i+1,j}x_{i +1}^{d-2}=	p_{j-1,i} x_j^{d-2}+(d-2)p_{ji}x_{j-1}x_j^{d-3}+p_{j+1,i}x_{j+1}^{d-2}.
	\end{equation}
When $d\geq 4$, the terms $x_{i-1}x_i^{d-3}$ do not appear in the right hand side of the above equality unless $i=j$. Hence $p_{ij}=0$ for $i \neq j$.  When $i=j+1$, the above equality also implies that $p_{jj}=p_{j+1,j+1}$. So the matrix $P$ is a scalar matrix and $f$ is a central form.

When $d=3$, the  equality $($\ref{hp}$)$ becomes $$p_{i-1,j}x_i+p_{ij}x_{i-1}+p_{i+1,j}x_{i+1}=p_{j-1,i} x_j+p_{ji}x_{j-1}+p_{j+1,i}x_{j+1}.$$
By comparing the coefficients of the $x_i$'s, we have
\begin{itemize}
	\item if $|i-j|\geq 3$, then $p_{ij}=0$;
	\item if $i=j-2$, then $ p_{j-1,j-2}=p_{j-2,j}=0$;
	\item if $i=j-1$, then $ p_{j-1,j}=p_{j+1,j-1}=0, \ p_{jj}=p_{j-1,j-1}$.
	%\item if $i=j+1$, then $p_{j-1,j+1}=p_{j+1,j}, p_{j+1,j+1}=p_{jj}$;
    %\item if $i=j+2$, then $p_{j+2,j}=p_{j+1,j+2}, p_{j,j+2}=p_{j+1,j }=0.$
\end{itemize}
Therefore, the matrix $P$ is a scalar matrix and $f$ is a central form.
\end{example}

Furthermore, we can show that almost all higher degree forms are central, so are \emph{a priori} indecomposable.

\begin{theorem}\label{mt}
Suppose $(n, d) \neq (2,3).$ Then the set of central forms is an open and dense subset of the affine space $\V.$
\end{theorem}

\begin{proof}
Suppose $f \in \V.$ We start with computing $Z(f)$ via Equation (\ref{ec}). It is well known that the matrix equations therein can be transformed to standard linear equations in the following way. Given an $n \times n$ matrix $X,$ let $X_{\v}$ be the $n^2$-dimensional column vector
\[
\left(
\begin{array}{c}
X_1 \\
\vdots \\
X_n
\end{array}
\right)
\]
where $X_i$ is the $i$-th column of $X.$ Let $X^T_{\v}$ denote the corresponding vector of $X^T$ and let $P$ be the permutation matrix such that $X^T_{\v} = P X_{\v}.$ Then the matrix equation $X^TA^{(i_3 \cdots i_d)}=A^{(i_3 \cdots i_d)}X$ is equivalent to the following linear equations
\begin{equation*}
 [I_n \otimes A^{(i_3 \cdots i_d)} -(A^{(i_3 \cdots i_d)} \otimes I_n)P]X_{\v} =0.
\end{equation*}
Let $B$ denote the $n^d \times n^2$ matrix of coefficients of the previous linear equations, that is
\begin{equation}\label{matrix}
B=
\left(
\begin{array}{ccc}
  I_n \otimes A^{(1 \cdots 11)} -(A^{(1 \cdots 11)} \otimes I_n)P \\
  I_n \otimes A^{(1 \cdots 1 2)} -(A^{(1 \cdots 12)} \otimes I_n)P \\
 \vdots\\
 I_n \otimes A^{(n \cdots n)} -(A^{(n \cdots n)} \otimes I_n)P
\end{array}
\right).
\end{equation}
Then clearly $Z(f)$ is obtained by putting the solution space of the linear equations
\begin{equation} \label{lec}
BX_{\v}=0
\end{equation}
back to the form of $n \times n$ matrices.

According to \eqref{lec}, a form $f$ is central if and only the associated matrix $B$ has rank $n^2-1.$ Therefore, the set of central forms in $\V$ is a union of all the principal open sets defined by the $n^2-1$ minors of $B.$ Moreover, this set is non-empty thanks to Example \ref{ecf}. It is well known that any non-empty Zariski open subset of $\V$ is dense. We are done.
\end{proof}

%Now it follows easily that
%
%\begin{corollary}
%The set of central forms is dense in $\V.$
%\end{corollary}
%In other words, almost all forms in $\V$ are central, a priori indecomposable.

%\begin{remark}%Comment Pumpl\"un's work \cite{pu}.
%As a direct consequence of the previous theorem, almost all higher degree forms are central, a priori indecomposable.
%\end{remark}

Next we consider the semisimplicity of the center algebras of higher degree forms. This turns out to be related to geometric properties of forms. Recall that a form $f(x_1, \cdots, x_n)$ is called smooth, or nonsingular, if the simultaneous equations \[\frac{\partial f }{ \partial x_1}= \cdots = \frac{\partial f }{ \partial x_n} =0\] have no nonzero solutions. There is also a notion of regularity of higher degree forms (see \cite{h2}) which generalizes and unifies the nondegeneracy and smoothness. That is, in terms of symmetric multilinear spaces, $(V,\Theta)$ is called $l$-regular if $\Theta(u, \cdots, u, v_{l+1}, \cdots, v_d)=0$ for all $v_{l+1}, \cdots, v_d \in V$ implies $u=0.$ Suppose $f$ is associated to $(V,\Theta)$ under the basis $e_1, \cdots, e_n.$ If $u=a_1 e_1+a_2e_2+\cdots+a_ne_n,$ then it can be checked (see \cite{h1}) that
\begin{equation} \label{par}
\Theta\left(u, \sum_{1 \le i \le n}x_ie_i, \cdots, \sum_{1 \le i \le n}x_ie_i\right)=\frac{1}{d} \sum_{1 \le i \le n} a_i \frac{\partial f}{\partial x_i}.
\end{equation}
With this it is clear that for corresponding symmetric $d$-linear spaces and higher degree forms $1$-regular = nondegenerate, $(d-1)$-regular = smooth and $l$-regular implies $(l-1)$-regular. In addition, if $f$ is non-$l$-regular, then $f$ has a singularity of multiplicity $d-l+1.$

It was noticed in \cite{h1, h2} that the condition of $2$-regularity for forms imposes very strong restriction, namely the semisimpleness, on their centers. For completeness, we include a proof here.

\begin{lemma}\label{lss}
Suppose $f$ is 2-regular. Then $Z(f)$ is semisimple.
\end{lemma}

\begin{proof}
It is enough to prove that if $\phi \in Z(f)$ is nilpotent, then $\phi=0.$ Otherwise, suppose there was a nonzero nilpotent element $\phi \in Z(f)$ with $\phi^{m+1}=0$ while $\phi^m \ne 0$ for some $m \ge 1.$ Then there must be some $v \in V$ such that $ \phi^m(v) \ne 0.$ Hence $$\Theta(\phi^m(v), \phi^m(v), v_3, \cdots, v_d)=\Theta(\phi^{2m}(v), v, v_3, \cdots, v_d)=\Theta(0, v, v_3, \cdots, v_d)=0$$ for all $v_3, \cdots, v_d \in V$ as $2m \ge m +1.$ Now the $2$-regularity of $f$ forces $\phi^m(v)=0.$ This leads to a desired contradiction.
\end{proof}

\begin{remark}
If a form is non-$2$-regular, then it has a singularity of multiplicity $d-1.$ There are non-$2$-regular higher degree forms whose center algebras are semisimple. For example, the determinant of a generic $n \times n$ matrix with $n \ge 3$ is non-$2$-regular as any rank $1$ matrix is a common zero of its all degree $n-2$ differentials. However the generic determinant has trivial center, namely $\k,$ see Example \ref{det} for more details.
\end{remark}

For a complete description of the semisimplicity of center algebras, we need the notion of limit of direct sums forms introduced in \cite{bbkt}.

\begin{definition}
A higher degree form $f$ is said to be a limit of direct sums (LDS) form if after a reversible linear change of variables, \begin{equation} \label{lds}
f(x_1, \cdots, x_n)=\sum_{i=1}^{l} x_i \frac{\partial h(x_{l+1},\cdots, x_{2l})}{\partial x_{l+i}}+ g(x_{l+1},\cdots, x_n),
\end{equation}
where $h$ and $g$ are forms of the same degree as $f,$ in $l$ and $n-l$ variables respectively.
\end{definition}

The terminology is justified by the following equation
\begin{eqnarray*}
  f(x_1, \cdots, x_n) &=& \lim_{t\rightarrow 0} \frac{1}{t}[h(tx_1+x_{l+1}, \cdots, tx_l+x_{2l})-h(x_{l+1},\cdots,x_{2l}) \\
   & & +tg(tx_1+x_{l+1}, \cdots, tx_l+x_{2l},x_{2l+1}, \cdots, x_n)]
\end{eqnarray*}
as the latter is a direct sum when $t \ne 0.$

\begin{theorem}\label{cns}
Suppose $f$ is a nondegenerate higher degree form. Then the center $Z(f)$ is nonsemisimple if and only if $f$ is an LDS form.
\end{theorem}

\begin{proof}
Suppose $f(x_1, \cdots, x_n)=\sum_{i=1}^{l} x_i \frac{\partial h(x_{l+1},\cdots, x_{2l})}{\partial x_{l+i}}+ g(x_{l+1},\cdots, x_n)$ is an LDS form. Then the Hessian matrix $H$ of $f$ is
$\left(
  \begin{array}{cc}
    0 & \begin{array}{cc}
          H_h & 0
        \end{array}
     \\
   \begin{array}{c}
     H_h \\
     0
   \end{array}
     & H_g \\
  \end{array}
\right)$ where $H_g$ and $H_h$ are the Hessian matrices of $g$ and $h$ respectively.
Let $N$ denote the block matrix
$\left(
\begin{array}{ccc}
0  & \mathrm{I}_l  & 0  \\
 0 &  0 & 0  \\
 0 &  0 &  0
\end{array}
\right)$ which is clearly nilpotent. It is easy to verify that $(HN)^T=HN.$ This implies that $Z(f)$ is nonsemisimple.

Conversely, suppose $Z(f)$ is nonsemisimple. Then take a nontrivial nilpotent element $\phi \in Z(f)$ satisfying $\phi^2=0.$ This is possible as $Z(f)$ is a commutative algebra. Let $(V, \Theta)$ be the associated symmetric $d$-linear space of $f.$ Assume $\rk \phi =l.$ Let $e_1, \cdots, e_l$ be a basis of $\Im \phi$ and choose $e_{l+i} \in V$ such that $\phi(e_{l+i})=e_i$ for all $1 \le i \le l.$ Then $e_1, \cdots, e_{2l}$ are linearly independent and we extend them to a basis $e_1, \cdots, e_n$ of $V.$ Note that $\phi(e_j)=0$ whenever $j \le l$ or $j \ge 2l+1.$ It follows that $\Theta(e_{i_1}, \cdots, e_{i_d})=0$ whenever there are two indices $i_s, i_t \in [1,l]$ or $i_s \in [1,l]$ and $i_t \in [2l+1, n].$ Then, under this basis, the form $f$ becomes
\begin{eqnarray*}
f(x_1, \cdots, x_n)&=&\Theta\left(\sum_{i=1}^n x_ie_i, \cdots, \sum_{i=1}^n x_ie_i\right) \\
&=&\sum_{i=1}^l x_i \sum_{1 \le i_2, \cdots, i_d \le l} \Theta(e_i, e_{l+i_2}, \cdots, e_{l+i_d}) x_{l+i_2} \cdots x_{l+i_d} \\
&& +\sum_{l+1 \le j_1, \cdots, j_d \le n} \Theta(e_{j_1}, \cdots, e_{j_d}) x_{j_1} \cdots x_{j_d}.
\end{eqnarray*}
Let $g(x_{l+1}, \cdots, x_n)=\sum_{l+1 \le j_1, \cdots, j_d \le n} \Theta(e_{j_1}, \cdots, e_{j_d}) x_{j_1} \cdots x_{j_d}$ and $$h_i(x_{l+1}, \cdots, x_{2l})=\sum_{1 \le i_2, \cdots, i_d \le l} \Theta(e_i, e_{l+i_2}, \cdots, e_{l+i_d}) x_{l+i_2} \cdots x_{l+i_d}$$ for all $1 \le i \le l.$ Note that $h_i=\frac{\partial f}{\partial x_i}$ and thus by \eqref{par} one has
\begin{eqnarray*}
\frac{\partial h_i}{\partial x_{l+j}}=\frac{\partial^2 f}{\partial x_i \partial x_{l+j}}&=&d(d-1)\Theta\left(e_i,e_{l+j},\sum_{k=l+1}^{2l} x_ke_k, \cdots, \sum_{k=l+1}^{2l} x_ke_k\right) \\
&=&d(d-1)\Theta\left(\phi(e_{l+i}),e_{l+j},\sum_{k=l+1}^{2l} x_ke_k, \cdots, \sum_{k=l+1}^{2l} x_ke_k\right) \\
&=&d(d-1)\Theta\left(e_{l+i},\phi(e_{l+j}),\sum_{k=l+1}^{2l} x_ke_k, \cdots, \sum_{k=l+1}^{2l} x_ke_k\right) \\
&=&d(d-1)\Theta\left(e_{l+i},e_j,\sum_{k=l+1}^{2l} x_ke_k, \cdots, \sum_{k=l+1}^{2l} x_ke_k\right) \\
&=&d(d-1)\Theta\left(e_j,e_{l+i},\sum_{k=l+1}^{2l} x_ke_k, \cdots, \sum_{k=l+1}^{2l} x_ke_k\right) \\
&=&\frac{\partial^2 f}{\partial x_j \partial x_{l+i}}=\frac{\partial h_j}{\partial x_{l+i}}.
\end{eqnarray*}
Then by the well known Euler's identity, the following degree $d$ form $$h(x_{l+1}, \cdots, x_{2l})=\frac{1}{d}\sum_{1 \le i \le l} x_{l+i}h_i(x_{l+1}, \cdots, x_{2l})$$ satisfies $\frac{\partial h}{\partial x_{l+i}}=h_i$ for all $1 \le i \le l.$ Now we have shown that \[f(x_1, \cdots, x_n)=\sum_{i=1}^{l} x_i \frac{\partial h(x_{l+1},\cdots, x_{2l})}{\partial x_{l+i}}+ g(x_{l+1},\cdots, x_n).\] That is to say, $f$ is an LDS form.
\end{proof}

Now we consider the direct sum decompositions of higher degree forms with semisimple centers.

\begin{theorem} \label{criterion}
Suppose $Z(f)$ is semisimple, or equivalently $f$ is not an LDS form. Then the number of indecomposable direct summands of $f$ is exactly $\dim Z(f).$
\end{theorem}

\begin{proof}
Under the assumption, the center algebra $Z(f)$ is isomorphic to $\k \times \cdots \times \k$ thanks to the well known Wedderburn-Artin Theorem. Then the assertion is a direct consequence of Proposition \ref{pc}.
\end{proof}

\begin{remarks}
Keep the assumption that $f$ is a nondegenerate higher degree form.
\begin{enumerate}
  \item If $f$ has semisimple center, then $\dim Z(f) \le n$ as the number of direct summands is not greater than the number of variables. In the specific case of $\dim Z(f)=n,$ the form $f$ is equivalent to the sum of $d$-th power of $n$ independent linear forms. These are the so-called diagonalizable forms and have been investigated in our previous paper \cite{hlyz} via the theory of Harrison's centers.
  \item It was observed in \cite{f1} that if $Z(f)$ is nonsemisimple, then $f$ is a nullform, or GIT unstable. The converse is not true. For example, the binary form $f=x^2y^3+xy^4+y^5$ is a nullform (see e.g. \cite{mukai}), however $Z(f) \cong \k.$ In other words, if the center algebra of a form has nontrivial nilpotent elements, then it is extremely singular. It would be interesting to unravel the connection between the nilpotency of $Z(f)$ and the singularity of $f.$
  \item In the literature, it is considered very difficult to understand the nature of central higher degree forms, see for example \cite{pu}. Now this seems reasonable, as a generic higher degree form is central due to our Theorem \ref{mt}. On the other hand, by Theorems \ref{cns} and \ref{criterion} we may interpret a central form as an indecomposable non LDS form.
\end{enumerate}
\end{remarks}

The previous theorem gives an elementary criterion for direct sum decomposability of higher degree forms with semisimple center.

\begin{corollary} \label{ind}
Suppose $Z(f)$ is semisimple. Then $f$ is indecomposable, i.e. not a direct sum, if and only if $\dim Z(f)=1,$ if and only if $\rk B=n^2-1$ where $B$ is as in \eqref{matrix}.
\end{corollary}

Note that smooth forms have semisimple centers by Lemma \ref{lss}, so our criterion contains the cases treated in \cite{bbkt,f1,w}. In comparison, our criterion is extremely simple as it is equivalent to computing the rank of a finite set of vectors. In addition, there is also a very simple algorithm for direct sum decompositions of higher degree forms with semisimple centers. Given a nondegenerate form $f,$ the first step is to solve the linear equations \eqref{lec} and take a basis of the solution space. The second step is to diagonalize simultaneously the chosen basis. Note that $Z(f)$ is semisimple if and only if each $X \in Z(f)$ is diagonalizable by the well known Jordan decomposition theorem. This can be detected via the minimal polynomial of any basis element, using Euclid's algorithm to see whether the minimal polynomial has multiple roots. This seems computationally cheaper than the Jacobian criterion of detecting the smoothness of $f.$ The third step is to find out diagonal idempotent matrices from linear combinations of the obtained set of diagonal matrices. Finally one determines a set of primitive orthogonal idempotents and decompose the form $f$ accordingly.

The algebraic structure of center algebras can easily provide a complete answer to the classical question of whether a higher degree form is reconstructible from its Jacobian ideal. Suppose $f \in \V.$ Let $J(f)$ be the Jacobian ideal generated by $\frac{\partial f}{\partial x_1},\cdots,\frac{\partial f}{\partial x_n}$ and let $E(f)$ be the vector space spanned by $\frac{\partial f}{\partial x_1},\cdots,\frac{\partial f}{\partial x_n}$. Recall that $f$ can be reconstructed from $J(f),$ or equivalently from $E(f),$ is equivalent to saying that if there is a $g \in \V$ such that $J(g)=J(f),$ or equivalently $E(g)=E(f),$ then $g = \lambda f$ for some $\lambda \in \k^*.$ The crux is the following observation appeared first in \cite{cg} to our knowledge, see also \cite{k}. For reader's convenience, we present it here with full detail.

\begin{lemma}\label{cj}
Suppose $f \in \V$ and $A\in\k^{n\times n}$. Then there exists a $g \in \V$ such that $$\left(\frac{\partial g}{\partial x_1},\cdots,\frac{\partial g}{\partial x_n}\right)=\left(\frac{\partial f}{\partial x_1},\cdots,\frac{\partial f}{\partial x_n}\right)A$$ if and only if $A\in Z(f)$.
\end{lemma}

\begin{proof}
It is well known that a set  of polynomials $\{g_1,\cdots, g_n\}$ can be lifted to a polynomial $g$ such that $\frac{\partial g}{\partial x_i}=g_i$ for all $i$ if and only if  $\frac{\partial g_i}{\partial x_j}=\frac{\partial g_j}{\partial x_i}$ for all $i, j.$ Let $g_i$ be the $i$-th coordinate of the vector $(\frac{\partial f}{\partial x_1},\cdots,\frac{\partial f}{\partial x_n})A$. Then it is easy to verify that the $g_i$'s can be lifted if and only if $HA$ is symmetric, that is, $A\in Z(f)$.
\end{proof}

\begin{theorem}\label{jacobian}
Suppose $f \in \V$ is a nondegenerate higher degree form. Then the following are equivalent.
\begin{enumerate}
  \item $f$ can be reconstructed from $J(f).$
  \item $f$ is a central form, i.e. $Z(f) \cong \k.$
  \item $\rk B=n^2-1$ where $B$ is as in \eqref{matrix}.
  \item $f$ is an indecomposable non LDS form.
\end{enumerate}
\end{theorem}

\begin{proof}
(2) $\Rightarrow$ (1) is obvious. It remains to prove (1) $\Rightarrow$ (2) as the rest are direct consequences of Theorems \ref{mt} and \ref{cns}. Assume $f$ can be reconstructed from $J(f).$ Let $A \in Z(f)$ be invertible and then let $g \in \V$ be the polynomial lifted by $\left(\frac{\partial f}{\partial x_1},\cdots,\frac{\partial f}{\partial x_n}\right)A.$ This is possible thanks to Lemma \ref{cj}. Then clearly $E(g)=E(f),$ so $g=\lambda f$ for some $\lambda \in \k^*.$ It follows that $A=\lambda I,$ i.e., a scalar matrix.
\end{proof}

In another words, if a nondegenerate form $f \in \V$ can not be reconstructed from $J(f),$ then $Z(f)$ must have either a nontrivial idempotent or a nontrivial nilpotent element, i.e. $f$ is either a direct sum or an LDS form. In particular, if $Z(f)$ has a nontrivial nilpotent element, then $f$ has a singularity of multiplicity $d-1$ due to Theorem \ref{cns} and Lemma \ref{lss}. As direct corollaries to Theorems \ref{cns}, \ref{criterion} and \ref{jacobian}, we can easily recover the main result of Wang \cite[Theorem 1.1]{w} and the related result of Fedorchuk \cite[Theorem 3.2]{f1}. Moreover, our criterion for the reconstructibility of a higher degree form from its Jacobian ideal amounts to computing the rank of a matrix.

Finally we provide an algorithm, purely in terms of linear algebra, for the direct sum decompositions of any higher degree forms. This can be obtained by slightly extending the algorithm for forms with semisimple centers, again by the Jordan decomposition theorem.

\begin{algorithm}\label{algo}
Take an arbitrary $f \in \V.$ Denote the associated symmetric tensor by $A.$

\emph{Step 1.} %Degeneracy test for any $f \in \V.$
Compute $\rk \{ A_1, \cdots, A_n \}.$ If it is $n,$ then $f$ is nondegenerate and continue; otherwise, take a linearly independent set of maximal size, reduce variables and make $f$ nondegenerate in lower dimension situation, then continue.

\emph{Step 2.} %Compute Z(f) and get a basis.
Solve the linear equations \eqref{lec} and get a basis $(P_i)_{1 \le i \le \dim Z(f)}$ of the center $Z(f).$ The $n^d \times n^2$ linear system is very special and can be vastly simplified. First, one can select a basis $H_1, \cdots, H_t$ of the subspace spanned by $\{A^{(i_3\cdots i_d)} \mid 1 \le i_3, \dots, i_d \le n \}$ and reduce the size of the system. Second, one can write $H_k=U_k^T D_k U_k$ where $D_k$ is diagonal and then highly simplify the associated linear equation.

\emph{Step 3.} %Take the semisimple part of Z(f) by Jordan decomposition.
Take the semisimple part $Q_i$ of each $P_i$ by its Jordan decomposition. It is well known that $Q_i$ is a polynomial of $P_i$ and so an element of $Z(f).$ Then diagonalize $(Q_i)_{1 \le i \le \dim Z(f)}$ simultaneously, and get a set of diagonal matrices $(\alpha_i)_{1 \le i \le \dim Z(f)}.$ Let $Z'$ be the $\k$-algebra generated by the $\alpha_i$'s.

\emph{Step 4.} %Compute Z(f)/Rad.
Determine a complete set of primitive orthogonal idempotents of $Z'$ which are linearly spanned by the $\alpha_i$'s. Write each $\alpha_i$ as a row vector and put them into a matrix $C.$ Then a set of primitive orthogonal idempotents are obtained by a row echelon reduction of $C.$ By the reverse simultaneous conjugation of Step 3, get a complete set of primitive orthogonal idempotents, denoted by $(\epsilon_j)_{1 \le j \le \dim Z'},$ of $Z(f).$

\emph{Step 5.} %Decompose $f$ accordingly.
Decompose the higher degree form $f$ according to the complete set $(\epsilon_j)_{1 \le j \le \dim Z'}$ of primitive orthogonal idempotents.
\end{algorithm}

\begin{remark}
The previous algorithm shows that the direct sum decompositions of higher degree forms can be boiled down to some standard tasks of linear algebra, specifically the computations of eigenvalues and eigenvectors for which there are efficient algorithms and well-established softwares.
%Our algorithm seems more elementary than those in some previous works \cite{bbkt, f1} which involve sophisticated tools of Gr\"obner bases and associated forms.
\end{remark}

\section{Examples}
In the following we provide some examples to elucidate the proposed criteria and algorithms. Examples 4.1 and 4.2 are taken form \cite{f1}. It turns out they can be easily worked out by hand without involving computers. Example 4.3 was considered in \cite{bbkt,f1,sh} by apolarity. Example 4.4 is about Cayley's hyperdeterminant and seems new. Example 4.5 is the Keet-Saxena cubic form which provides the best upper bound of the dimension of centers so far. These seemingly very complicated forms can all be handled by our approach of looking at the center without difficulty. Note that the examples are over not necessarily algebraically closed ground fields. Thus our approach works as well for direct sum decompositions of forms over arbitrary fields.

\begin{example}\cite[Example 6.4]{f1}
Consider the rational form $f(x_1, x_2, x_3)= x_1^3 +3x_1^2x_2 + 3x_1x_2^2 + 2x_2^3 + 3x_1^2x_3 + 6x_1x_2x_3 +4x_2^2x_3 + 3x_1x_3^2 +4x_2x_3^2 + 2x_3^3 \in \Q[x_1, \ x_2, \ x_3].$ Let
\[ A^{(1)}=
\left(
\begin{array}{ccc}
 1 & 1  & 1  \\
 1 &  1 & 1  \\
 1 & 1  &  1
\end{array}
\right), \quad
 A^{(2)}=
\left(
\begin{array}{ccc}
1  & 1  & 1  \\
1  & 2  & \frac{4}{3}  \\
 1 & \frac{4}{3}  &  \frac{4}{3}
\end{array}
\right), \quad
A^{(3)}=
\left(
\begin{array}{ccc}
1  & 1  & 1  \\
1  &  \frac{4}{3} & \frac{4}{3}  \\
 1 & \frac{4}{3}  &   2
\end{array}
\right).
\]
Then by direct calculation, $Z(f)=\{ X \in \Q^{3 \times 3} \mid A^{(i)}X=X^TA^{(i)}, \ 1 \le i \le 3 \}=\bigoplus_{1 \le i \le 3} \Q X_i,$ where
\[ X_1=
\left(
\begin{array}{ccc}
 1 & 1  & 1  \\
 0 &  0 & 0  \\
 0 &  0 & 0
\end{array}
\right), \quad
 X_2=
\left(
\begin{array}{ccc}
0  & -1  & -1  \\
 0 &  1 & 0   \\
 0 &  0 & 1
\end{array}
\right), \quad
X_3=
\left(
\begin{array}{ccc}
0  & 1  & 3  \\
 0 &  0 & 1  \\
 0 &  -1 & -4
\end{array}
\right).
\]
It is clear that $X_1$ and $X_2$ are a pair of orthogonal idempotents. According to the proof of Proposition \ref{pc}, take the change of variables \[ y_1=x_1+x_2+x_3, \quad y_2=x_2, \quad y_3=x_3 \] and decompose the form as \[y_1^3+(y_2^3+y_2^2y_3+y_2y_3^2+y_3^3).\] Let $g(y_2,y_3)=y_2^3+y_2^2y_3+y_2y_3^2+y_3^3.$ Then one can read from the $X_i$'s that
$$Z(g)=\Q
\left(
\begin{array}{cc}
1 & 0 \\
0 & 1 \\
\end{array}
\right)
\oplus \Q
\left(
\begin{array}{cc}
0 & 1 \\
-1 & -4 \\
\end{array}
\right) \cong \Q[\sqrt{3}].$$ It follows by Proposition \ref{pc} that $g$ is indecomposable over $\Q$ as $\Q[\sqrt{3}]$ is a field. However it is not absolutely indecomposable. Over any field extension $K/\Q$ with $\sqrt{3} \in K,$ one has easily $Z(g) \otimes _\Q K \cong K \times K$ and $g$ can be further decomposed as
\[
2\left(\frac{3+\sqrt{3}}{6}y_2+\frac{3-\sqrt{3}}{6}y_3\right)^3 + 2\left(\frac{3-\sqrt{3}}{6}y_2 +\frac{3+\sqrt{3}}{6}y_3\right)^3 .
\]
To summarize, $f$ can be diagonalized over $K$ as
\[ f(x_1,x_2,x_3)=\left(x_1 + x_2 +x_3\right)^3 + 2\left(\frac{3+\sqrt{3}}{6}x_2+\frac{3-\sqrt{3}}{6}x_3\right)^3 + 2\left(\frac{3-\sqrt{3}}{6}x_2 +\frac{3+\sqrt{3}}{6}x_3\right)^3 .\]
\end{example}

\begin{example}\cite[Example 6.6]{f1}
Consider the following rational quaternary quartic
\begin{equation*}
\begin{aligned}
f=&x_1^4+4x_1^3x_2+6x_1^2x_2^2+4x_1x_2^3+2x_2^4+8x_1^3x_3+24x_1^2x_2x_3+24x_1x_2^2x_3+8x_2^3x_3+24x_1^2x_3^2\\
&+48x_1x_2x_3^2+24x_2^2x_3^2+32x_1x_3^3+32x_2x_3^3+17x_3^4-12x_1^3x_4-36x_1^2x_2x_4-36x_1x_2^2x_4\\
&-12x_2^3x_4-72x_1^2x_3x_4-144x_1x_2x_3x_4-72x_2^2x_3x_4-144x_1x_3^2x_4-144x_2x_3^2x_4-96x_3^3x_4\\
&+54x_1^2x_4^2+108x_1x_2x_4^2+54x_2^2x_4^2+216x_1x_3x_4^2+217x_2x_3x_4^2+216x_3^2x_4^2-108x_1x_4^3\\
&-108x_2x_4^3-216x_3x_4^3+82x_4^4.
\end{aligned}
\end{equation*}
Let $A$ be the associated symmetric tensor of $f$ and for all $1 \le i_3, i_4 \le 4$ let $A^{(i_3i_4)}=(a_{i_1i_2i_3i_4})_{1 \le i_1, i_2 \le 4}.$ Note that
\begin{eqnarray*}
  A^{(22)}-A^{(11)} = \left(
                          \begin{array}{cccc}
                            0 & 0 & 0 & 0 \\
                            0 & 1 & 0 & 0  \\
                            0 & 0 & 0 & 0  \\
                            0 & 0 & 0 & 0  \\
                          \end{array}
                        \right),&
   \quad
  12A^{(23)}-24A^{(11)} = \left(
                              \begin{array}{cccc}
                                0 & 0 & 0 & 0 \\
                                0 & 0 & 0 & 0  \\
                                0 & 0 & 0 & 0  \\
                                0 & 0 & 0 & 1  \\
                              \end{array}
                            \right),&
   \\
  12A^{(24)}+36A^{(11)} = \left(
                              \begin{array}{cccc}
                                0 & 0 & 0 & 0 \\
                                0 & 0 & 0 & 0  \\
                                0 & 0 & 0 & 1  \\
                                0 & 0 & 1 & 0  \\
                              \end{array}
                            \right),&
  \quad
  A^{(33)}-4A^{(11)}=  \left(
                              \begin{array}{cccc}
                                0 & 0 & 0 & 0 \\
                                0 & 0 & 0 & 0  \\
                                0 & 0 & 1 & 0  \\
                                0 & 0 & 0 & 0  \\
                              \end{array}
                            \right),&
  \\
  12A^{(34)}+72A^{(11)}= \left(
                              \begin{array}{cccc}
                                0 & 0 & 0 & 0 \\
                                0 & 0 & 0 & 1  \\
                                0 & 0 & 0 & 0  \\
                                0 & 1 & 0 & 0  \\
                              \end{array}
                            \right),&
  \quad
  12A^{(44)}-144A^{(23)}+180A^{(11)}=\left(
                              \begin{array}{cccc}
                                0 & 0 & 0 & 0 \\
                                0 & 0 & 1 & 0  \\
                                0 & 1 & 0 & 0  \\
                                0 & 0 & 0 & 0  \\
                              \end{array}
                            \right).&
\end{eqnarray*}
With the preceding observation it is easy to compute that $Z(f)=\{ X \in \Q^{4 \times 4} \mid A^{(ij)}X=X^TA^{(ij)}, \ 1 \le i,j \le 4 \} =\Q X_1 \oplus \Q X_2$ where
\[
X_1=\left(
  \begin{array}{cccc}
    1 & 1 & 2 & -3 \\
    0 & 0 & 0 & 0 \\
    0 & 0 & 0 & 0 \\
    0 & 0 & 0 & 0 \\
  \end{array}
\right),
\quad
X_2=\left(
      \begin{array}{cccc}
        0 & -1 & -2 & 3 \\
        0 & 1 & 0 & 0 \\
        0 & 0 & 1 & 0 \\
        0 & 0 & 0 & 1 \\
      \end{array}
    \right).
\]
Clearly $X_1$ and $X_2$ are a pair of primitive orthogonal idempotents and $Z(f) \cong \Q \times \Q.$ As the previous example, we have the following direct sum decomposition
\[ f=(x_1+x_2+2x_3-3x_4)^4+(x_2^4+x_3^4+x_2x_3x_4^2+x_4^4). \]
In conclusion, $f$ is the direct sum of two absolutely indecomposable forms over $\Q.$
\end{example}

\begin{example}[Determinant-like polynomials] \label{det}
Let $n \ge 3.$ Consider the determinant-like polynomial
\[ D:=\sum_{\sigma \in S_n} c_{\sigma} x_{1, \sigma(1)} x_{2, \sigma(2)} \cdots x_{n, \sigma(n)}, \]
where $c_\sigma \in \k^*.$ Denote $D_{ij,kl}:=\frac{\partial^2 D}{\partial x_{ij} \partial x_{kl}}.$ Arrange the indeterminates $x_{ij}$ lexicographically via their indices. Assume $A=(a_{st, uv}) \in Z(D).$ Take an arbitrary pair of indices $ij$ and $kl$ with $i \ne k, \ j \ne l.$ Then by \eqref{hessian} we have
\[ \sum_{st} D_{ij,st} a_{st,kl} = \sum_{uv} D_{kl,uv} a_{uv,ij}. \] Note that $D_{ij,st}=D_{kl,uv} \ne 0$ if and only if $s=k, \ t=l, \ u=i, \ v=j.$ It follows that $a_{ij,ij}=a_{kl,kl}$ and $a_{st,kl}=0$ for any other $st$ with $s \ne i, \ t \ne j.$ By varying $ij$ and $kl$, it can be checked that $A$ is a scalar matrix. That is, $Z(D) \cong \k$ and thus $D$ is absolutely indecomposable. One can also consider pfaffian-like and hafnian-like polynomials in the same manner.
\end{example}

\begin{example}[Cayley's hyperdetermiant]
The well known Cayley's hyperdetermiant (of $2\times2\times2$ matrix) is the following $8$-variate degree $4$ form:
\begin{equation*}
\begin{array}{ccl}
f(x_1,\cdots,x_8)&=&x_1^2x_8^2+x_2^2x_7^2+x_3^2x_6^2+x_4^2x_5^2\\
                 &&-2(x_1x_2x_7x_8+x_1x_3x_6x_8+x_1x_4x_5x_8+x_2x_3x_6x_7+x_2x_4x_5x_7+x_3x_4x_5x_6)\\
                 &&+4(x_1x_4x_6x_7+x_2x_3x_5x_8).
\end{array}
\end{equation*}
Let $A=(a_{i_1 i_2 i_3 i_4})$ be the associated symmetric $4$-tensor of $f$ and let $A^{(i,j)}$ be the $8\times 8$ matrix $(a_{i_1 i_2 i j})_{1 \le i_1, i_2 \le 8}.$ By $E_{i,j}$ we denote the $8 \times 8$ matrix with $(i,j)$-entry $1$ and other entries $0.$ Suppose $T=(T_{ij}) \in Z(f).$ We have the following observations.

\begin{itemize}
  \item $A^{(i,i)}=E_{9-i,9-i}$ for all $1 \le i \le 8.$ By \eqref{hessian}, it follows easily that the matrix $T$ is diagonal, written simply as $\operatorname{diag}(t_1,t_2,\dots,t_8)$;
  \item $A^{(1,j)}=-\frac{1}{6}(E_{9-j,8}+E_{8,9-j})$ for $j=2,3,4$. It follows that $t_5=t_6=t_7=t_8$;
  \item $A^{(k,8)}=-\frac{1}{6}(E_{9-k,1}+E_{1,9-k})$ for $j=4,6,7$. It follows that $t_1=t_2=t_3=t_5$;
  \item $A^{(5,8)}=-\frac{1}{6}(E_{4,1}+E_{1,4})+\frac{1}{3}(E_{3,2}+E_{2,3})$. It follows that $t_1=t_4$.
\end{itemize}
Thus the matrix $T$ must be a scalar matrix. Hence Cayley's hyperdetermiant is a central form and so absolutely indecomposable.
\end{example}

%\begin{example}[Elementary symmetric functions]
%
%
%\end{example}

\begin{example}[The Keet-Saxena cubic forms \cite{keet, s}]
Let $f=\sum\limits_{i=1}^{n} a_{ii}X_i^2+\sum\limits_{1\leq i<j\leq n}2a_{ij}X_iX_j$ be a polynomial in the indeterminates $a_{ij},X_i$ with $1\leq i \leq j\leq n$. Then $f$ is a cubic form in $\frac{n(n+3)}{2}$  variables. Compute the second-order partial derivatives of $f$ to get
\[
\frac{\partial^2f}{\partial a_{ij}\partial a_{i'j'}}=0, \quad
\frac{\partial^2f}{\partial X_i \partial X_j}=2a_{ij}, \quad
\frac{\partial^2f}{\partial a_{ij} \partial X_k}=\left\{
	\begin{aligned}
	0,&  & \ \ k\neq i,j; \\
	2X_j, &  & k=i; \\
	2X_i, &  & k=j.
	\end{aligned}
	\right.
\]
Arrange the $a_{ij}$'s by lexicographic order. Then the Hessian matrix of $f$ looks like
$$H=
	\left(
	\begin{array}{c c}
	0& X \\
	X^T& A
	\end{array}
	\right)$$
where $X$ is an $\frac{n(n+1)}{2}\times n$ matrix, $X^T$ is the transpose of $X$ and $A$ is the generic $n\times n$ symmetric matrix.
For any $P \in Z(f),$ partition it
$\left(
	\begin{array}{c c}
	P_1& P_2 \\
	P_3& P_4
	\end{array}
	\right)$
as $H$. As $HP=\left(
	\begin{array}{c c}
	XP_3& XP_4 \\
	X^tP_1+AP_3& X^tP_2+AP_4
	\end{array}
	\right)$
is symmetric and $A$ is generic, it follows easily that $P_3=0$ and $\left(
	\begin{array}{c c}
	P_1& 0\\
	0& P_4
	\end{array}
	\right)$ is a scalar matrix.
Note that the $(k_1, k_2)$-entry of $X^tP_2$ is
$$2(X_1P_{1k_1, k_2}+\cdots+X_{k_1}P_{k_1k_1, k_2}+X_{k_1+1}P_{k_1\ k_1+1, k_2}+\cdots +X_nP_{k_1n,
k_2}),$$ where $P_2=(P_{ij, k})_{1\leq i \leq j\leq n, \ 1\leq k\leq n}$. Then $X^tP_2$ is symmetric if and only if $$\sum\limits_{i\leq k_1} P_{ik_1,k_2}X_i+\sum\limits_{i>k_1} P_{k_1i,k_2}X_i=\sum\limits_{i\leq k_2} P_{ik_2,k_1}X_i+\sum\limits_{i>k_2}P_{k_2i,k_1}X_i$$ for all $1\leq k_1 \ne k_2 \leq n.$ By comparing the coefficients of $X_i'$s, we have that
\begin{itemize}
	\item if $i\leq \min\{k_1,k_2\}$, then $P_{ik_1,k_2}=P_{ik_2,k_1}$;
	\item if $k_1<i<k_2$, then $P_{k_1i,k_2}=P_{ik_2,k_1}$;
	\item if $k_2<i<k_1$, then $P_{ik_1,k_2}=P_{k_2i,k_1}$;
	\item if $i\geq \max\{k_1,k_2\}$, then $P_{k_1 i,k_2}=P_{k_2 i,k_1}$.	
\end{itemize}
Therefore, we have  $P_{i_1j_1,k_1}=P_{i_2j_2,k_2}$ if and only if the triples $\{i_1,j_1,k_1\}=\{i_2,j_2,k_2\}$ counting with the multiplicities, and if and only if $X_{i_1}X_{j_1}X_{k_1}=X_{i_2}X_{j_2}X_{k_2}$ as monomials. It is well known that the number of degree $3$ monomials in $n$ variables is $n+2 \choose 3$. Hence we have $\mathrm{dim} Z(f)=1+C_{n+2}^3=1+\frac{(n+2)(n+1)n}{6}$. Moreover, $Z(f)$ is a local algebra with square zero radical. Therefore, the Keet-Saxena cubic $f$ is an indecomposable LDS form.
\end{example}

\begin{remark}
It was conjectured in \cite{os1} that $\dim Z(f) \le n$ for all nondegenerate $f \in \V.$ By Theorem \ref{mt}, the dimension conjecture of O'Ryan-Shapiro holds for almost all nondegenerate forms. Note that the Keet-Saxena cubic forms have dimension of magnitude $O(n^{\frac{3}{2}}),$ so are counter examples. In \cite{s} Saxena also tried to amend the O'Ryan-Shapiro conjecture as: $\dim Z(f) \le (d-1)n$ for all nondegenerate $f \in \V.$ As a matter of fact, this is also disproved by the same Keet-Saxena cubic forms.
\end{remark}

\end{document}